\documentclass[10pt]{amsart}
\usepackage{color, graphicx}
\usepackage[square,sort,comma,numbers]{natbib}
\numberwithin{equation}{section} \overfullrule 5pt
\usepackage{amsmath,amssymb}

\newdimen\tableauside\tableauside=1.0ex
\newdimen\tableaurule\tableaurule=0.4pt
\newdimen\tableaustep
\def\phantomhrule#1{\hbox{\vbox to0pt{\hrule height\tableaurule
width#1\vss}}}
\def\phantomvrule#1{\vbox{\hbox to0pt{\vrule width\tableaurule
height#1\hss}}}
\def\sqr{\vbox{%
  \phantomhrule\tableaustep

\hbox{\phantomvrule\tableaustep\kern\tableaustep\phantomvrule\tableaustep}%
  \hbox{\vbox{\phantomhrule\tableauside}\kern-\tableaurule}}}
\def\squares#1{\hbox{\count0=#1\noindent\loop\sqr
  \advance\count0 by-1 \ifnum\count0>0\repeat}}
\def\tableau#1{\vcenter{\offinterlineskip
  \tableaustep=\tableauside\advance\tableaustep by-\tableaurule
  \kern\normallineskip\hbox
    {\kern\normallineskip\vbox
      {\gettableau#1 0 }%
     \kern\normallineskip\kern\tableaurule}%
  \kern\normallineskip\kern\tableaurule}}
\def\gettableau#1 {\ifnum#1=0\let\next=\null\else
  \squares{#1}\let\next=\gettableau\fi\next}

\newcommand{\be}{\begin{equation}}
\newcommand{\ee}{\end{equation}}
\newcommand{\bea}{\begin{eqnarray}}
\newcommand{\eea}{\end{eqnarray}}
\newcommand{\ba}{\begin{array}}
\newcommand{\ea}{\end{array}}

\newcommand{\id}{\hbox{1\kern-.27em l}}

\newcommand{\la}{\lambda}

\newcommand{\cN}{\mathcal{N}}

\newcommand{\non}{\nonumber}

\newcommand{\SO}{\mathrm{SO}}

\newcommand{\Sp}{\mathrm{Sp}}

\newcommand{\Spin}{\mathrm{Spin}}

\newtheorem{Lem}{Lemma}
\newtheorem{Pro}{Proposition}

\newtheorem{remark}{Remark}[section]

\newtheorem{thm}{Theorem}[section]

\newtheorem{lem}[thm]{Lemma}

\theoremstyle{definition}
\allowdisplaybreaks

\title{Fingerprint Invariant  of  Partitions and   Construction}
\author[Bao. Shou]{Bao. Shou$^{\dag}$}
\author[Qiao Wu]{Qiao Wu$^*$}

\address{$^*$ College of Logistics and E-commerce \\Zhejiang Wanli University\\
No.8 South Qianhu Road, Ningbo 315100, P.R.China}
\address{$\dag$Center  of Mathematical  Sciences\\Zhejiang University \\
Hangzhou 310027,  China}

\email{$^{\dag}$ bsoul@zju.edu.cn, \quad $^*$10920005@zju.edu.cn}



\begin{document}
\begin{abstract} 
The fingerprint invariant of partitions can be   used to describe the Kazhdan-Lusztig map for the classical groups.  We discuss the basic properties of  fingerprint. We construct   the fingerprints of rigid partitions in the $B_n$, $C_n$, and $D_n$ theories. To calculate the fingerprint of a rigid semisimple operator $(\lambda^{'};\lambda^{''})$, we decompose $\lambda^{'}+\lambda^{''}$ into several blocks.  We define operators to calculate the fingerprint  for each block  using the results of fingerprint of the unipotent operators.
\end{abstract}
\maketitle

\tableofcontents

\section{Introduction}
The fingerprint invariant of partitions  is used to describe the Kazhdan-Lusztig map for the classical groups \cite{Lusztig:1984}\cite{Symbol 2}. The Kazhdan-Lusztig map  is a map from the unipotent conjugacy classes to the set of conjugacy classes of the Weyl group. It can be extended to the case of rigid semisimple conjugacy classes\cite{Spaltenstein:1992}. The rigid semisimple conjugacy classes and  the conjugacy classes of the Weyl group are described by pairs of partitions $(\lambda^{'};\lambda^{''})$   and pairs of partitions  $[\alpha;\beta]$, respectively. The fingerprint invariant  is  a map between these two classes of objects.

Gukov and Witten  initiated to  study the $S$-duality for surface operators in $\cN=4$ super Yang-Mills theories in the ramified case of the Geometric Langlands Program in \cite{Gukov:2006}\cite{Witten:2007}\cite{Gukov:2008}. Surface operators are labeled by   pairs of certain partitions.  The $S$-duality conjecture suggest that surface operators in  the theory with gauge group $G$ should have a counterpart  in the Langlands dual group $G^{L}$. The {\it 'rigid'} surface operators  corresponding to pairs of rigid partitions are expected to be closed under $S$-duality. Fingerprint  is an invariant for the dual pairs partitions.

There is another invariant of partitions called symbol related to the Springer correspondence\cite{Collingwood:1993}.  By using  symbol invariant, Wyllard made some explicit proposals for how the $S$-duality maps should act on rigid surface operators in \cite{Wyllard:2009}. In \cite{Shou 1:2016}, we find a new subclass of rigid surface operators related by S-duality.  In \cite{Shou 2:2016}, we found a construction  of  symbol for the rigid partitions in the  $B_n, C_n$, and $D_n$ theories. Compared to the symbol invariant, the calculation of fingerprint is a complicated and tedious work. In this paper, we attempt to continue  the analysis in \cite{Gukov:2006}\cite{Wyllard:2009}\cite{Shou 1:2016} \cite{Shou 2:2016}, studying  the construction of fingerprint.

With noncentral rigid conjugacy classes in $A_n$  and  more complicated exceptional groups,  we will concentrate on the $B_n, C_n$, and $D_n$ series in this paper.   The fingerprint invariant is   assumed to be  equivalent to the symbol invariant \cite{Wyllard:2009}. Hopefully our constructions would  be helpful to prove  the equivalence\cite{iv}. Clearly more work is required.

The following is an outline of  this article.  In Section 2, we  introduce  some basic results related to the rigid partition and concept of fingerprint.  In Section 3, we calculate the fingerprint of the unipotent operators in the $B_n, C_n$, and $D_n$ theories. In Section 4, we calculate the fingerprint of rigid semisimple operators $(\lambda^{'};\lambda^{''})$. We decompose $\lambda=\lambda^{'}+\lambda^{''}$ into several blocks.   For $I$ type block,   we define  operators $\mu_{e1}$, $\mu_{e2}$ and their odd versions $\mu_{o1}$, $\mu_{o2}$. For $II$ type blocks,   we define  operator $\mu_{II}$. For $III$ type blocks,   we define  operators $\mu_{e11}$, $\mu_{e12}$, $\mu_{e21}$,  $\mu_{e22}$ and their odd versions $\mu_{o11}$, $\mu_{o12}$, $\mu_{o21}$,  $\mu_{o22}$. For $S$ type blocks, the fingerprint should be calculated from case to case.

\section{Fingerprint of rigid partitions}
In this section, we introduce rigid partitions in the $B_n$, $C_n$, and $D_n$ theories.  Then  we introduce the definition of  fingerprint\cite{Wyllard:2009}.
\subsection{Rigid Partitions in the  $B_n, C_n$, and $D_n$ theories}
A partition $\la$ of the positive integer $n$ is a decomposition $\sum_{i=1}^l \la_i = n$  ($\la_1\ge \la_2 \ge \cdots \ge \la_l$), with the length  $l$. Partitions are in a one to one correspondence  with Young tableaux.  For instance the partition $3^22^31$  corresponds to
\be
\tableau{2 5 7}\non
\ee
Young tableaux occur in a number of branches of mathematics and physics. They  are also  tools for the construction of  the eigenstates of Hamiltonian System \cite{Shou:2011} \cite{{Shou:2014}} \cite{{Shou:2015}} and  label  the fixed point of the localization of path integral\cite{Localization}.


For the $B_n$($D_n$)theories, unipotent conjugacy  classes  are in one-to-one correspondence with partitions of $2n{+}1$($2n$) where all even integers appear an even number of times.  For the $C_n$ theories, unipotent conjugacy  classes  are in one-to-one correspondence with partitions $2n$ for which all odd integers appear an even number of times.  If it has no gaps (i.e.~$\la_i-\la_{i+1}\leq1$ for all $i$) and no odd (even) integer appears exactly twice, a partition in the $B_n$ or $D_n$ ($C_n$) theories is called {\it rigid}. We will focus on rigid partition in this paper.

The following facts  plays an important role in this study.
\begin{Pro}{\label{Pb}}
The longest row in a rigid  $B_n$ partition always contains an odd number of boxes. The following two rows of the first row are either both of odd length or both of even length.  This pairwise pattern then continues. If the Young tableau has an even number of rows the row of shortest length has to be even.
\end{Pro}
\begin{flushleft}
\textbf{Remark:} If the last row of the partition is odd, the number of rows is odd.
\end{flushleft}

\begin{Pro}{\label{Pd}}
For a rigid $D_n$ partition, the longest row  always contains an even number of boxes. And the following two rows are either both of even length or both of odd length. This pairwise pattern then continue. If the Young tableau has an even number of rows the row of the shortest length has to be even.
\end{Pro}
Examples of partitions in the $B_n$ and  $D_n$ theories are shown in Fig.(\ref{BD}).
     \begin{figure}[!ht]
  \begin{center}
    \includegraphics[width=4.8in]{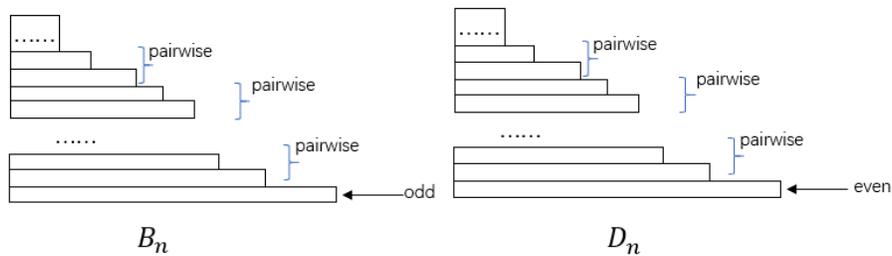}
  \end{center}
  \caption{ Partitions in the $B_n$    and  $D_n$  theories. }
  \label{BD}
\end{figure}

\begin{Pro}{\label{Pc}}
For a rigid $C_n$ partition, the longest two rows  both contain either  an even or an odd number  number of boxes.  This pairwise pattern then continues. If the Young tableau has an odd number of rows the row of shortest length has contain an even number of boxes.
\end{Pro}
Example of partitions in the $C_n$ theory are shown in Fig.(\ref{C}).
\begin{figure}[!ht]
  \begin{center}
    \includegraphics[width=2.2in]{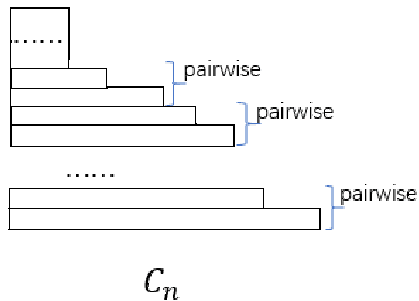}
  \end{center}
  \caption{ Partitions in the $C_n$ theory }
  \label{C}
\end{figure}

\subsection{Definition of fingerprint}
We introduce the definition of fingerprint.  First, add the two partitions  $\lambda=\lambda^{'}+\lambda^{"}$, and then calculate   the partition $\mu=Sp(\lambda)$
\begin{eqnarray}\label{mu}
\mu_i=Sp(\lambda)_i=
\left\{ \begin{aligned}
         & \lambda_i + p_{\lambda}(i) \quad  \quad   \textrm{if} \quad \lambda_i  \textrm{ is odd  and} \quad \lambda_i \neq \lambda_{i-p_{\lambda}(i)}, \\
         & \lambda_i   \quad  \quad   \quad   \quad  \quad        \textrm{ otherwise}
       \end{aligned} \right.
\end{eqnarray}
Next define the function $\tau$ from an even positive integer $m$ to $¡À1$ as follows. For a partition in the $B_n$ and $D_n$ theories,  $\tau(m)$ is $-1$ if  at least one $\mu_i$ such that $\mu_i= m$ and either of the following three conditions is satisfied.
\begin{eqnarray}\label{con}
  &&(i)\quad\quad\quad\quad\quad \,\,\mu_i\neq \lambda_i  \non \\
 && (ii) \quad\quad\quad\sum^{i}_{k=1}\mu_k \neq \sum^{i}_{k=1}\lambda_k\\
&& (iii)_{SO} \quad\quad\quad \lambda^{'}_{i} \quad\textrm{is odd}.  \non
\end{eqnarray}
Otherwise  $\tau$ is 1.
For the  partitions in the $C_n$ theory,  the definition is exactly  the same except the condition $(iii)_{SO}$ is  replaced by
$$(iii)_{Sp}\quad\quad \lambda^{'}_{i}  \quad   \textrm{is even}.$$
Finally  construct  a pair of partitions $[\alpha;\beta]$, which is related to  the conjugacy classes of the Weyl group. For each pair of parts of $\mu$ both equal to $a$, satisfying $\tau(a)=1$,   retain one part $a$ as a part of  the partition $\alpha$. For each part of $\mu$ of size $2b$, satisfying $\tau(2b)=-1$,  retain $b$ as a part of  the partition $\beta$.

Note that $\mu_i\neq\lambda_i$ only happen at the end of a row.

\subsection{Preliminary}
For a partition $\lambda=m^{n_m}{(m-1)}^{n_{m-1}} {(m-2)}^{n_{m-2}} \cdots {1}^{n_1} $ as shown in Fig.(\ref{p1}),
according to Propositions {\ref{Pb}},{\ref{Pd}}, and  {\ref{Pc}},    the number of boxes of the  gray part  is even.
\begin{figure}[!ht]
  \begin{center}
    \includegraphics[width=2.5in]{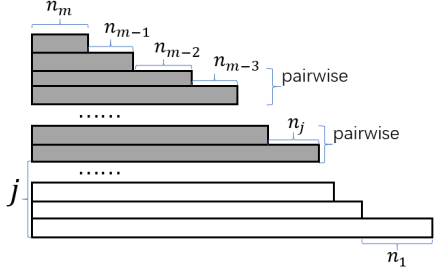}
  \end{center}
  \caption{  Number of boxes of the gray parts  is even.}
  \label{p1}
\end{figure}
For the partition in the $B_n$ and $D_n$ theories, $j$ is odd. If  the $\frac{j-1}{2}$th pairwise rows  are even, the number of boxes of the gray parts  in  Fig.(\ref{p2})(a) and (b)  are even.    If  the $\frac{j-1}{2}$th pairwise rows  are odd, the number of boxes of the gray parts  in  Fig.(\ref{p2})(a) and (b)  are odd.  Since $j$ is even for the partitions in the $C_n$ theory,  the numbers of boxes of the gray parts in Fig.(\ref{p2})(a) and (b) are always even.
\begin{figure}[!ht]
  \begin{center}
    \includegraphics[width=4.5in]{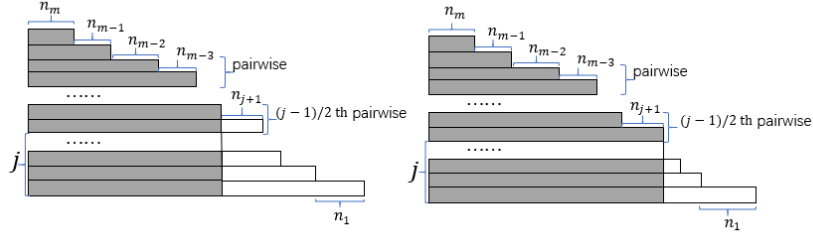}
  \end{center}
  \caption{ Parity of the  number of boxes of  the gray parts depends on the parity of the height $j$. }
  \label{p2}
\end{figure}

Next, we calculate the fingerprint for the first few parts  of a $B_n$ or $D_n$ partition as an example. If $n_a$ is the first odd number  in the sequence $n_m,n_{m-1},\cdots,n_j$,  then the $m, m-1,\cdots,{a+1}$th rows are even.
The $a$th row is odd, which is the second row of a pairwise rows.
 Let  $L(l)=\sum^{l}_{i=1}n_i$. The $L(l)$th column  is at the end  of the $l$th row as shown in Fig.(\ref{p2}).
So we have
$$p_{L(i)}=(-1)^{\sum_{k=1}^{L(i)}\lambda_k}=1,\quad\quad i\in (m,\cdots, a-1),$$
which means
$$\mu_k=\lambda_k, \quad\quad k\in (1,\cdots ,L(a)-1).$$
 The height of the $a$th row is $a$, which is odd. So we have
$$p_{L(a)}=(-1)^{\sum_{k=1}^{L(a)}\lambda_k}=-1.$$
 According to  (\ref{mu}),   $\mu_{L(a)}=\lambda_{L(a)}-1$,  satisfying  the condition (\ref{con})$(i)$ for the part $\lambda_{L(a)}(=a-1)$. So we have
$\tau(a-1)=-1$ for the  even number $a-1$.
The  $a-1$th row is the first row of a  pairwise rows, so $n_{a-1}$ is even. And we have
$$p_{L(a-1)}=(-1)^{\sum_{k=1}^{L(a)}\lambda_k}\cdot(-1)^{\sum_{k=L(a)}^{L(a-1)}\lambda_k}=(-1)^{\sum_{k=1}^{L(a)}\lambda_k}\cdot(-1)^{\sum_{k=L(a)+1}^{L(a)+n_{a-1}}(a-1)}=-1,$$
which means $\mu_{L(a-1)+1}=\lambda_{L(a-1)+1}+1$. Since $\mu_i\neq\lambda_i$ only happen at the end of a row,  we have
$$\mu_k =\lambda_k, \quad\quad  i\in (L(a)+1,\cdots, L(a)+n_{a-1}-1 ).$$
Because of  $\mu_{L(a)}=\lambda_{L(a)}+1$,  we have
\begin{equation}\label{fe}
  \sum^{L(a-1)+1}_{k=1}\mu_k = \sum^{L(a-1)+1}_{k=1}\lambda_k.
\end{equation}

For a rigid semisimple operator $(\lambda^{'}, \lambda^{"})$, following the above  procedure,    we can    construct  the partition $\mu$  for the parts $k^{n_k}$ independently.  Let
$'-'$ and $'+'$ denote  $p_\lambda(i)=(-1)^{\sum^{i}_{k=1}\lambda_k}=-1$ and  $1$, respectively. The results are as  shown in Fig.(\ref{OR}) and Fig.(\ref{ER}).  The gray box is always  appended as  the last part of row  and the black box is deleted at  the end of  row.
If $k$ is odd, there are four cases according to  $p_\lambda(L(k+1))$  and the parity of $n_k$  as shown in Fig.(\ref{OR}).
\begin{itemize}
  \item For Fig.(\ref{OR})(a), we have   $p_\lambda(L(k+1))=-1$. Then $p_\lambda(L(k+1)+1)=1$, which means $\mu_{L(k+1)+1}=\lambda_{L(k+1)+1}+1$. Since $n_k$ is even, we have $p_\lambda(L(k))=-1$ according to Fig.(\ref{p2}), which means $\mu_{L(k)}=\lambda_{L(k)}-1$.
  \item For Fig.(\ref{OR})(b), we have   $p_\lambda(L(k+1))=-1$. Then $p_\lambda(L(k+1)+1)=1$, which means $\mu_{L(k+1)+1}=\lambda_{L(k+1)+1}+1$. Since $n_k$ is odd, we have $p_\lambda(L(k))=1$ according to Fig.(\ref{p2}), which means $\mu_{L(k)}=\lambda_{L(k)}$.
  \item For Fig.(\ref{OR})(c), we have   $p_\lambda(L(k+1))=1$. Then $p_\lambda(L(k+1)+1)=-1$, which means $\mu_{L(k+1)+1}=\lambda_{L(k+1)+1}$. Since $n_k$ is odd, we have $p_\lambda(L(k))=-1$ according to Fig.(\ref{p2}), which means $\mu_{L(k)}=\lambda_{L(k)}-1$.
  \item For Fig.(\ref{OR})(a), we have   $p_\lambda(L(k+1))=1$. Then $p_\lambda(L(k+1)+1)=-1$, which means $\mu_{L(k+1)+1}=\lambda_{L(k+1)+1}$. Since $n_k$ is even, we have $p_\lambda(L(k))=1$ according to Fig.(\ref{p2}), which means $\mu_{L(k)}=\lambda_{L(k)}$.
\end{itemize}
 If $k$ is even, there are two cases according to  $p_\lambda(L(k+1))$  and the parity of $n_k$  as shown in Fig.(\ref{ER}).
\begin{itemize}
  \item For Fig.(\ref{ER})(a), we have $p_\lambda(L(k+1))=-1$, which means $\mu_{L(k+1)}=\lambda_{L(k+1)}-1$. Since $k$ is even, we have $p_\lambda(L(k))=-1$ according to Fig.(\ref{p2}), which means $\mu_{L(k)+1}=\lambda_{L(k)+1}+1$.
  \item For Fig.(\ref{ER})(b), we have   $p_\lambda(L(k+1))=1$. Then $p_\lambda(L(k+1)+1)=1$, which means $\mu_{L(k+1)+1}=\lambda_{L(k+1)+1}$. Since $k$ is even, we have $p_\lambda(L(k))=1$ according to Fig.(\ref{p2}), which means $\mu_{L(k)}=\lambda_{L(k)}$.
\end{itemize}
According to the above results, we have the following important lemma. Without confusion, the image of the map $\mu$  is alsod called the partition $\mu$.
\begin{Lem}\label{ff}
Under the map $\mu$, the change of the last box of a row  of a partition depend on  the parity of the row and  the sign of $p_\lambda(i)$.
\begin{center}
  \begin{tabular}{|c|c|l|}
  \hline
  Parity of row & Sign & Change \\
  odd & $-$ & $\mu_i=\lambda_i-1$ \\
  even & $-$ & $\mu_i=\lambda_i+1$ \\
  even & $+$ & $\mu_i=\lambda_i$ \\
  odd & $+$ & $\mu_i=\lambda_i$ \\
  \hline
\end{tabular}
\end{center}
\end{Lem}

  \begin{figure}[!ht]
  \begin{center}
    \includegraphics[width=4.8in]{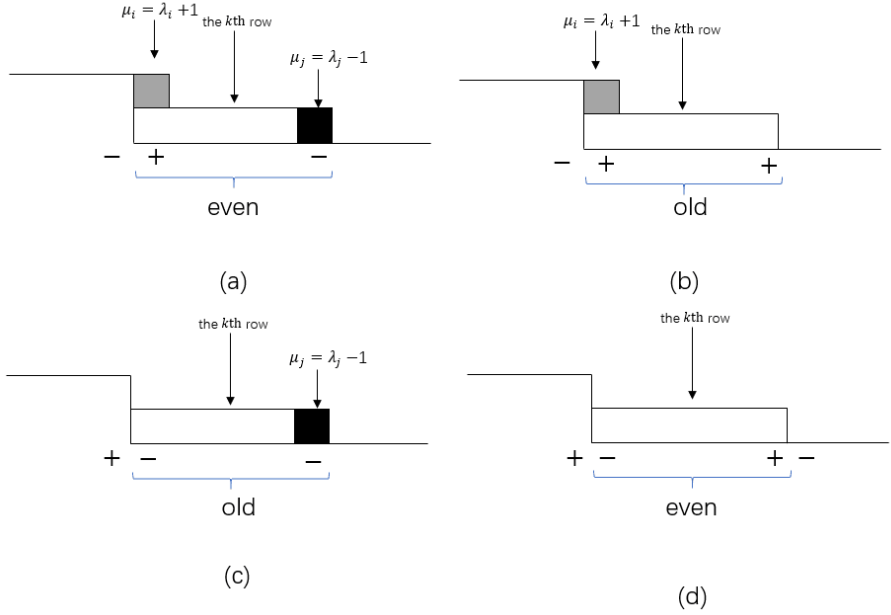}
  \end{center}
  \caption{ Partition $\mu$  for  the parts $k^{n_k}$($k$ is odd). }
  \label{OR}
\end{figure}

\begin{figure}[!ht]
  \begin{center}
    \includegraphics[width=4.8in]{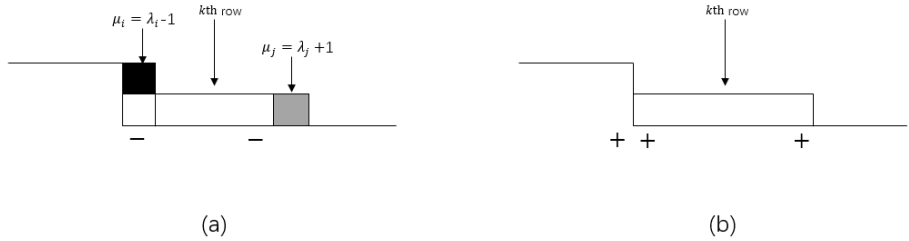}
  \end{center}
  \caption{  Partition $\mu$  for  the parts $k^{n_k}$( $k$ is even). }
  \label{ER}
\end{figure}

The partition  $\mu$ of the partition $\lambda=\lambda^{'}+\lambda^{"}$  can be constructed by using the   building blocks in  Fig.(\ref{OR}) and Fig.(\ref{ER}).  It is consist of several cycles as shown in Fig.(\ref{M}).
\begin{figure}[!ht]
  \begin{center}
    \includegraphics[width=4.8in]{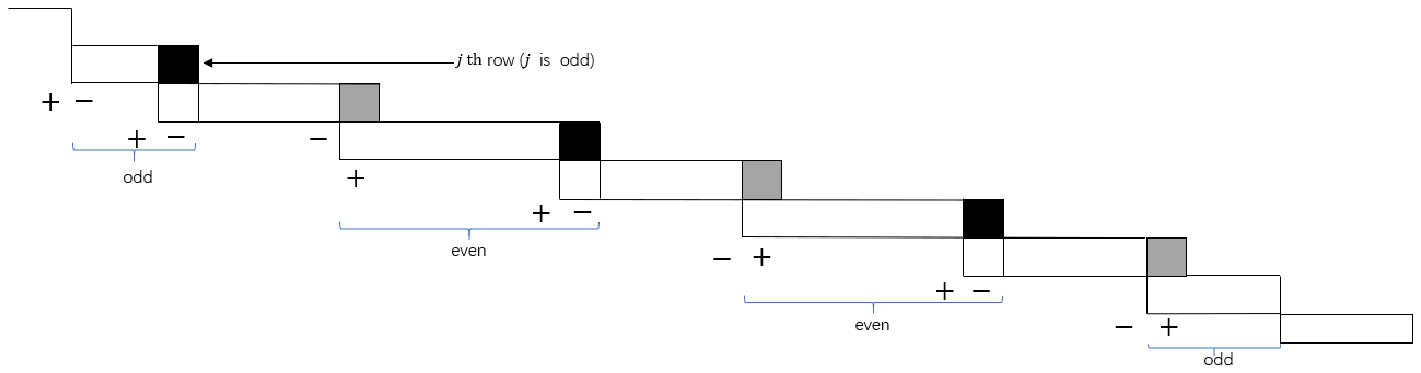}
  \end{center}
  \caption{ One cycle of partition $\mu$ between  successive odd number $n_j$ and $n_i$.}
  \label{M}
\end{figure}
 Both $i$ and $j$ are odd and both $n_i$ and $n_j$  are odd. This cycle begin with Fig.(\ref{OR})(c) corresponding to   $n_j$     and  end with    Fig.(\ref{OR})(b) corresponding to   $n_i$.              We can decompose  this cycle into building blocks in Fig.(\ref{OR}) and Fig.(\ref{ER}) as follows
\begin{eqnarray*}
  Fig.(\ref{OR})(c)\rightarrow Fig.(\ref{ER})(a)\rightarrow && Fig.(\ref{OR})(a) \rightarrow  \nonumber  \\
    Fig.(\ref{ER})(a) \rightarrow && Fig.(\ref{OR})(c)\rightarrow Fig.(\ref{ER})(a)\rightarrow Fig.(\ref{OR})(b).
\end{eqnarray*}

We introduce several facts to  characterise the definition of fingerprint. The following lemma  has been illustrated by Eq.(\ref{fe}).
\begin{lem}
The condition  $ \sum^{i}_{k=1} \mu_i \neq \sum^{i}_{k=1} \lambda_i $ in the definition of fingerprints  must   follow the  condition  $  \mu_j \neq \lambda_j $  until another $  \mu_j \neq \lambda_j $.
\end{lem}

 \begin{lem}
The effect of the operation $Sp(\lambda)$ is to ensure that the odd parts of the resulting partition never occur an odd number of times.
 \end{lem}

For the semisimple operator  $(\lambda^{'},\lambda^{''})$ in the $B_n$ theory, $\lambda^{'}$ and   $\lambda^{''}$  are partitions in the $B_n$ and $D_n$ theories, respectively. Since  the first row of the partitions in the $B_n$ and $D_n$ theories is  not in  pairwise rows, $\lambda^{'}$ and   $\lambda^{''}$ can be seen as in  the same theory.
\begin{lem}
When $\lambda=\lambda^{'}+\lambda^{''}$ is even, the condition
$$(iii)_{SO}, \quad \lambda^{'}_{i} \quad \textrm{is odd} $$
is equivalent to the condition
$$(iii)_{SO}, \quad \lambda^{''}_{i} \quad \textrm{is odd }$$.
\end{lem}

\begin{lem}
If the fingerprint of  $\lambda=\lambda^{'}+\lambda^{''}$  is $[\alpha,\beta]$, the fingerprints of the partition with  $\lambda_i+2$ correspond to $\alpha_i+2$ and $\beta_i+1$.
\end{lem}
\begin{proof}  $\lambda_i \rightarrow \lambda_i+2$ correspond to  $\mu_i \rightarrow \mu_i+2$,  which lead to the conclusion.
\end{proof}

We have the following fact from Fig.(\ref{M}).
\begin{Pro}
The condition (\ref{con})$(i)$  imply $(ii)$ for partition $\lambda$ with $\lambda_i-\lambda_{i+1}\leq 1$.
\end{Pro}
We would give an example,  which the condition $(ii)$ works as shown in Fig.(\ref{con2}). Both the height of the $(j-1)$th row, the $(j-3)$th row, and the $(j-5)$th are even. And  the difference of the height of the $(j-1)$th row and the $(j-3)$th row violate the rigid condition $\lambda_i-\lambda_{i+1}\leq 1$ as well as height of the $(j-3)$th row and the $(j-5)$th row. The part $(j-3)$ satisfy the condition (\ref{con})$(ii)$ but do not satisfy the condition $(i)$.
\begin{figure}[!ht]
  \begin{center}
    \includegraphics[width=2.5in]{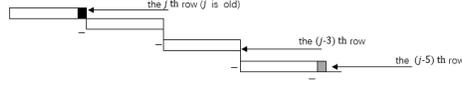}
  \end{center}
  \caption{ The difference of the height of the $(j-1)$th row and the $(j-3)$th row is two, violating the rigid condition $\lambda_i-\lambda_{i+1}\leq 1$.}
  \label{con2}
\end{figure}

\section{Fingerprint of rigid unipotent operators}\label{fu}
For rigid unipotent operators corresponding to  unipotent conjugacy classes, the calculations of  fingerprint can be simplified  using Propositions (\ref{Pb}), (\ref{Pd}), and (\ref{Pc}).
 The partition $\mu$  can be described by  maps $X_S$ and $Y_S$.

\subsection{$B_n$ and $D_n$ theories}\label{bdfinger}
First, we introduce the  map $X_S$  and the map $Y_S$. The map $X_S$ map a partition with only odd rows in the $B_n$ theory to  a partition with only even rows in the $C_n$ theory as shown in  Fig.(\ref{xs}).
\begin{eqnarray} \label{XS}
X_S:&& m^{2n_m+1}\, (m-1)^{2n_{m-1}}\, (m-2)^{2n_{m-2}  } \cdots 2^{2n_2} \, 1^{2n_1}  \non \\   & \mapsto&
m^{2n_m}\, (m-1)^{2n_{m-1} +2 }\, (m-2)^{2n_{m-2} - 2 } \cdots 2^{n_2+2} \, 1^{2n_1-2}\,.
\end{eqnarray}
\begin{figure}[!ht]
  \begin{center}
    \includegraphics[width=2.5in]{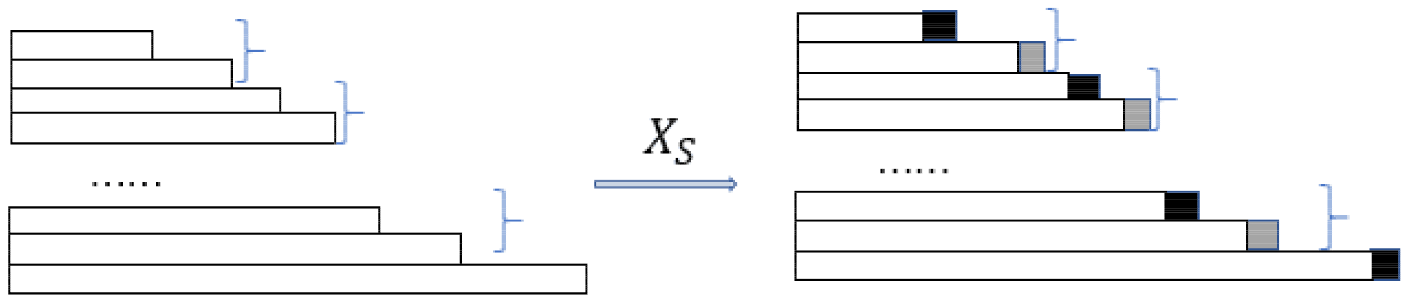}
  \end{center}
  \caption{ Map $X_S$. }
  \label{xs}
\end{figure}
where $m$ has to be odd in order for the first object to be a partition in the $B_n$ theory.  The black box of the second row of a pairwise rows is deleted and a gray  box is appended  at the end of the first row of a pairwise rows.  The $2k$th and $(2k+1)$th rows are in a pairwise rows of the partition in the $B_n$ theory. The $(2k+1)$th and $(2k+2)$th rows are in a pairwise rows of a partition in the $C_n$ theory under the map $X_S$.
The map $X_S$  is a bijection so that $X_S^{-1}$ is well defined.
It (\ref{XS}) is essentially the '$p_C$ collapse' described in \cite{Collingwood:1993} and the inver map  $X_S^{-1}$ is essentially the '$p^B$ expansion'.

\begin{figure}[!ht]
  \begin{center}
    \includegraphics[width=2.5in]{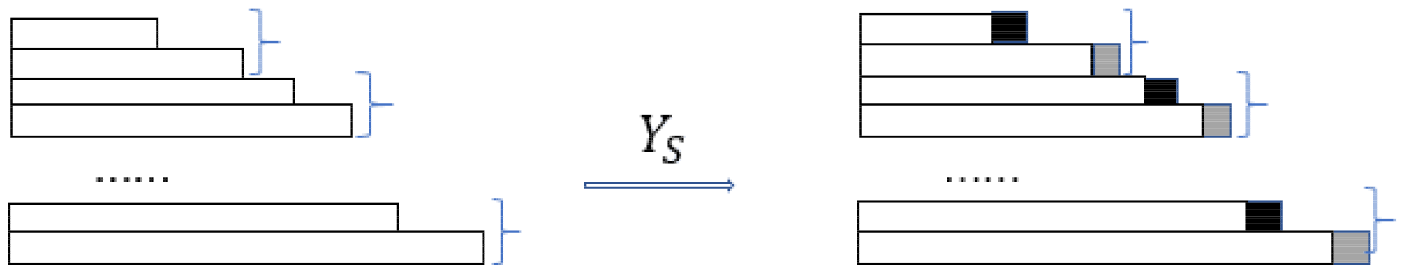}
  \end{center}
  \caption{ Map $Y_S$.  }
  \label{ys}
\end{figure}
Second, we propose the following map $Y_S$ which take  a  rigid partition with only odd rows  in the $C_n$ theory to  a  rigid partition with only even rows  in the $D_n$ theory as shown in Fig.(\ref{ys}).
\begin{eqnarray} \label{YS}
Y_S:&& m^{2n_m+1}\, (m-1)^{2n_{m-1}}\, (m-2)^{2n_{m-2}  } \cdots 2^{2n_2} \, 1^{2n_1}  \non \\   & \mapsto&
m^{2n_m}\, (m-1)^{2n_{m-1} +2 }\, (m-2)^{2n_{m-2} - 2 } \cdots 2^{n_2-2} \, 1^{2n_1+2}\,
\end{eqnarray}
where $m$ has to be even in order for the first element to be a $C_k$ partition.
The $2k+1$th and $(2k+2)$th rows are in a pairwise rows of a partition in the  $C_n$ theory. The $(2k+2)$th and $(2k+3)$th rows are in a pairwise rows of a partition in the $D_n$ theory under the map $Y_S$.
This is a bijection.   The map (\ref{XS}) is essentially the '$p_D$ collapse' described in \cite{Collingwood:1993} and the inver map  $Y_S^{-1}$ is essentially the '$p^C$ expansion'.

\begin{flushleft}
\textbf{Partition} $\mu$
\end{flushleft}
Since $\mu_i\neq\lambda_i$ only happen at the end of a row and the number boxes of each pairwise rows is even, we consider the image of each pairwise rows under the map $\mu$ independently.
First, consider an odd pairwise rows  of the partition in the $B_n$ or $D_n$ theories  as shown in Fig.(\ref{OO}). The $i$th row is the second row of the  pairwise rows, so $i$ is odd.  According to Fig.(\ref{p2})(a), we have
$$p_{\lambda}(L(i))=(-1)^{\sum^{L(i)}_{k=1}\lambda_k}=-1,$$
which means
$$\mu_{L(i)}=\lambda_{L(i)}+p_{\lambda}(L(i))=\lambda_{L(i)}-1.$$
Similarly, according to Fig.(\ref{p2})(b), we have
$$p_{\lambda}(L(i-1))=(-1)^{\sum^{L(i-1)}_{k=1}\lambda_k}=-1,$$
which means
$$p_{\lambda}(L(i-1)+1)=(-1)^{\sum^{L(i-1)}_{k=1}\lambda_k + \lambda_{L(i-1)+1} }=1.$$
Then we have
$$\mu_{L(i-1)+1}=\lambda_{L(i-1)+1}+p_{\lambda}(L(i-1)+1)=\lambda_{L(i-1)+1}+1.$$
We can also get this result using Lemma \ref{ff} directly. So  we reach  the right hand side of  Fig.(\ref{OO}). In fact,  this result  is Fig.(\ref{ER})(a).
\begin{figure}[!ht]
  \begin{center}
    \includegraphics[width=4.9in]{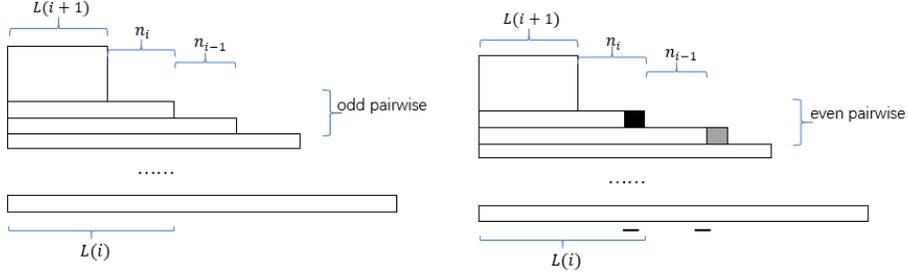}
  \end{center}
  \caption{ Partition $\mu$  for  the parts ${k^{n_k}(k-1)}^{n_{k-1}}$ corresponding an odd pairwise rows  of a partition in the $B_n$ or $C_n$ theory.}
  \label{OO}
\end{figure}

Next consider an even pairwise rows  of  the partition as shown in Fig.(\ref{EE}).  According to Fig.(\ref{p2})(a), we have
$$p_{\lambda}(L(i))=(-1)^{\sum^{L(i)}_{k=1}\lambda_k}=1,$$
which means
$$\mu_{L(i)}=\lambda_{L(i)}.$$
Similarly, according to Fig.(\ref{p2})(b), we have
$$p_{\lambda}(L(i-1))=(-1)^{\sum^{L(i-1)}_{k=1}\lambda_k}=1,$$
and then
$$p_{\lambda}(L(i-1)+1)=(-1)^{\sum^{L(i-1)}_{k=1}\lambda_k + \lambda_{L(i-1)+1} }=-1,$$
which means
$$\mu_{L(i-1)+1}=\lambda_{L(i-1)+1}.$$
We can also get this result using Lemma \ref{ff} directly. So  we reach  the right hand side of  Fig.(\ref{EE}).  In fact,  this  result is Fig.(\ref{ER})(b).
\begin{figure}[!ht]
  \begin{center}
    \includegraphics[width=4.9in]{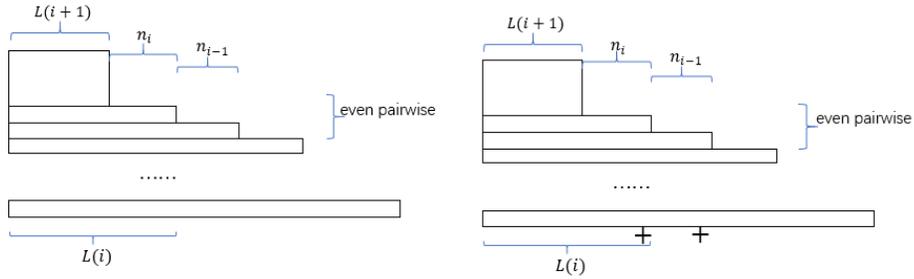}
  \end{center}
  \caption{ Partition $\mu$  for  the parts ${k^{n_k}(k-1)}^{n_{k-1}}$ corresponding an even pairwise rows  of a partition in the $D_n$ theory.. }
  \label{EE}
\end{figure}

The first row of a $B_n$ or $D_n$ partition is not in a pairwise rows. Since the number of boxes of a $B_n$ partition is odd, $p_{\lambda}=-1$ for the last part, which means we should delete the last box of the fist row according to Lemma \ref{ff}.  Since the number of boxes of a $D_n$ partition is even, $p_{\lambda}=1$ for the last part, which means nothing change for the first row  according to Lemma \ref{ff}.
Summary,  for a $B_n$ partition $\rho$, the $\mu$ partition is
$$\mu=X_S\rho_{odd}+\rho_{even}.$$
For a $D_n$ partition $\rho$, the $\mu$ partition is
$$\mu=Y_S\rho_{odd}+\rho_{even}.$$

\begin{flushleft}
\textbf{Fingerprint}
\end{flushleft}
\begin{figure}[!ht]
  \begin{center}
    \includegraphics[width=4.9in]{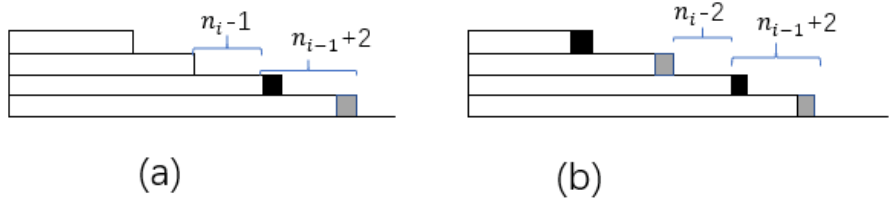}
  \end{center}
  \caption{ $(a)$ Image of an odd pairwise rows  following odd pairwise rows    under the map $\mu$.  $(b)$ Image of an odd pairwise rows   following even pairwise rows   under the map $\mu$. }
  \label{bdfo}
\end{figure}
Now, we construct fingerprint $[\alpha, \beta]$. The condition '$(iii)_{SO} \quad \lambda^{'}_{i} \quad\textrm{is odd}$'   can not be  satisfied for  even parts.  The pairwise rows  above an odd pairwise rows can be  an even pairwise rows or an odd pairwise  rows   as shown in Fig.(\ref{bdfo}). We compute the fingerprint for the parts $i^{n_i}(i-1)^{n_{i-1}}$. For Fig.(\ref{bdfo})(a), both $i$ and $n_i$ are odd  since the pairwise rows following   are even.  The fingerprint of the parts $i^{{n_i-1}}$ of $\mu$  is
\begin{equation}\label{bdfe1}
  i^{\frac{n_i-1}{2}}.
\end{equation}
Since the part $i-1$ satisfy the condition \ref{con} $i$, the fingerprint of the parts $(i-1)^{{n_{i-1}}+2}$  is
\begin{equation}\label{bdfe01}
 ({\frac{i-1}{2}})^{{n_{i-1}+2}}.
\end{equation}
For Fig.(\ref{bdfo})(a), the fingerprint is denoted as
$$F_{oe}=[i^{\frac{n_i-1}{2}},({\frac{i-1}{2}})^{{n_{i-1}+2}}].$$

 For Fig.(\ref{bdfo})(b),  $n_i$ is even since the pairwise rows following   are odd.   The fingerprint of the odd parts $i^{{n_i-2}}$  is
\begin{equation}\label{bdfe2}
  i^{\frac{n_i-2}{2}}.
\end{equation}
Since the part $i-1$ satisfy the condition \ref{con} $i$, the fingerprint of the parts $(i-1)^{{n_{i-1}}+2}$  is
\begin{equation}\label{bdfe02}
 ({\frac{i-1}{2}})^{{n_{i-1}+2}}.
\end{equation}
For Fig.(\ref{bdfo})(b), the fingerprint is denoted as
$$F_{oo}=[i^{\frac{n_i-2}{2}},({\frac{i-1}{2}})^{{n_{i-1}+2}}].$$
For both cases, (\ref{bdfe01})  is equal to (\ref{bdfe02}).


\begin{figure}[!ht]
  \begin{center}
    \includegraphics[width=4.9in]{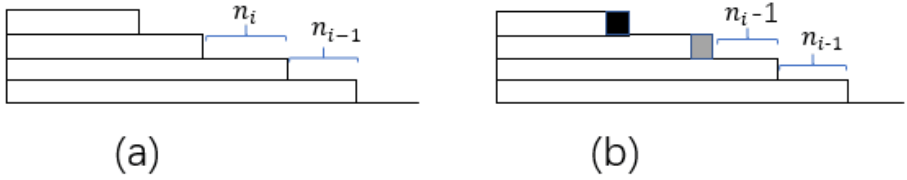}
  \end{center}
  \caption{ $(a)$, Image of an even pairwise rows  following even pairwise rows    under the map $\mu$.  $(b)$, Image of an even pairwise rows   following even pairwise rows   under the map $\mu$. }
  \label{bdfe}
\end{figure}
The pairwise rows  above an even pairwise rows can be  an even pairwise rows or an odd pairwise  rows  as shown in Fig.(\ref{bdfe}). For Fig.(\ref{bdfe})(a), $n_i$ is even since the pairwise rows following   are even.   The fingerprint of the odd parts $i^{{n_{i}}}$  is
\begin{equation}\label{bdfo1}
  i^{\frac{n_i}{2}}.
\end{equation}
And the fingerprint of the even parts $(i-1)^{{n_{i-1}}}$  is
\begin{equation}\label{bdfo01}
 ({{i-1}})^{\frac{n_{i-1}}{2}}.
\end{equation}
 For Fig.(\ref{bdfe})(b), $n_i$ is odd since the pairwise rows following   are odd.  The fingerprint of the odd parts $i^{{n_i}}$  is
\begin{equation}\label{bdfo2}
  i^{\frac{n_i-1}{2}}.
\end{equation}
And the fingerprint of the even parts $(i-1)^{{n_{i-1}}}$  is
\begin{equation}\label{bdfo02}
 ({{i-1}})^{\frac{n_{i-1}}{2}}.
\end{equation}
For both cases the formula  (\ref{bdfo01})is equal to the formula  (\ref{bdfo02}). Both the formula  (\ref{bdfo1})and the formula  (\ref{bdfo02}) can be denoted as  $i^{[\frac{n_i}{2}]}$.   Combined  the formulas (\ref{bdfo1}),(\ref{bdfo01}),(\ref{bdfo2}),and (\ref{bdfo02}), the fingerprint of the parts ${i^{n_i}(i-1)}^{n_{i-1}}$ is
\begin{equation}\label{bfi}
 F_{o}=[i^{[\frac{n_i}{2}]}({{i-1}})^{\frac{n_{i-1}}{2}}, \emptyset].
\end{equation}

Finally, we discuss the fingerprint of the parts $1^{n_1}$ in the $B_n$ and $D_n$ theories.
Since the number of the boxes of a $B_n$ partition is odd, the sign of the last part of the first row is $'-'$, which means a box is deleted under the map $\mu$. We can refer to the change of the  second row of a odd pairwise rows as shown in Fig.(\ref{bdfo}). If the row following the first row is even, according to  Fig.(\ref{bdfo})(a), the fingerprint is
$$[ 1^{\frac{n_1-1}{2}},\quad].$$
If the row following the first row is odd,  according to  Fig.(\ref{bdfo})(b),  the fingerprint is
$$[ 1^{\frac{n_1-2}{2}},\quad].$$

Since the number of the boxes of a $D_n$ partition is even, the sign of the last part of the first row is $'+'$, which means nothing changes under the map $\mu$. We can refer to the change of the  second row of a odd pairwise rows as shown in Fig.(\ref{bdfe}).  According to  Fig.(\ref{bdfe})(a) and Fig.(\ref{bdfe})(b), the fingerprint is
\begin{equation}\label{fbfbfbfb}
 [ 1^{\frac{[n_1]}{2}},\quad].
\end{equation}

\subsection{$C_n$ theory}\label{cfinger}
\begin{figure}[!ht]
  \begin{center}
    \includegraphics[width=3.8in]{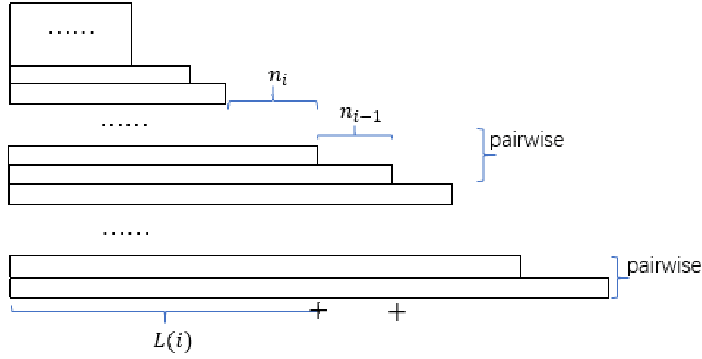}
  \end{center}
  \caption{ A partition in the $C_n$ theory.}
  \label{cc0}
\end{figure}
Different from the partitions in the $B_n$ and $D_n$ theories, the first two rows of a $C_n$ partition are of a pairwise rows as shown in Fig.(\ref{cc0}).     If $i$ and $(i-1)$th rows are of a pairwise rows, according to Fig.(\ref{f2}), $p_{\lambda}(L(i))=+1$ and  $p_{\lambda}(L(i-1))=+1$ for  $\lambda_{L(i)}$ and $\lambda_{L(i-1)}$ at  the end  of row, respectively.  Using  Lemma \ref{ff},  we have
$$\mu=\lambda,$$
as shown in Fig.(\ref{CC}).
\begin{figure}[!ht]
  \begin{center}
    \includegraphics[width=4.9in]{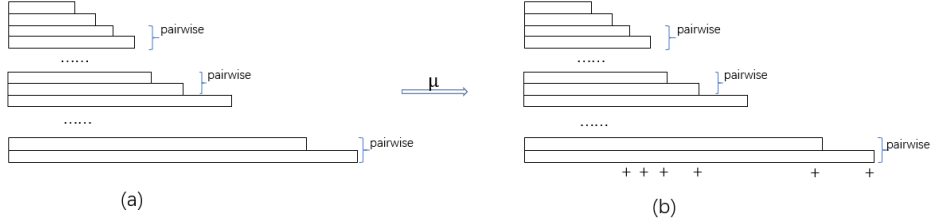}
  \end{center}
  \caption{  Images of the partition $\lambda$ in the $C_n$ theory under the map $\mu$.}
  \label{CC}
\end{figure}

The condition '$(iii)_{SO} \quad\quad\lambda^{'}_{i} \quad\textrm{is odd}$'  can not be    satisfied for  even parts.  Since $(i-1)$th row is the first row of a pairwise, $(i-1)$ is odd and $n_{i-1}$ is even. The parts ${(i-1)}^{{n_{i-1}}}$ contribute to the first factor of the fingerprint as follows
\begin{equation}\label{fi1}
  [{(i-1)}^{\frac{n_{i-1}}{2}}, \quad]
\end{equation}
 Since $i$th row is the second row of a pairwise, $i$ is even.
 The parts ${(i)}^{{n_{i}}}$ contribute to the second factor of the fingerprint as follows
 \begin{equation}\label{fi2}
   [ {{i}}^{\frac{n_{i}}{2}}, \quad].
 \end{equation}
Combining the  formulas   (\ref{fi1}) and (\ref{fi2}), the fingerprint of $\lambda$ can be written  down directly as follows
\begin{equation}\label{cfi}
  [\prod_{i=1} {{i}}^{\frac{n_{i}}{2}}, \emptyset].
\end{equation}

\section{Fingerprint of rigid semisimple   operators}
In this section, we propose an algorithm to calculate the fingerprints of rigid semisimple   operators $(\lambda^{'}, \lambda^{''})$ in the $B_n$, $C_n$, and $D_n$ theories. We decompose $\lambda=\lambda^{'}+\lambda^{''}$ into several blocks whose fingerprints can be constructed   by using the results of rigid unipotent operators. The fingerprint of the $C_n$ rigid semisimple operators  is a  simplified  version of that in the $B_n$ and $D_n$ theories.

\subsection{$B_n$ theory}
\begin{figure}[!ht]
  \begin{center}
    \includegraphics[width=4.9in]{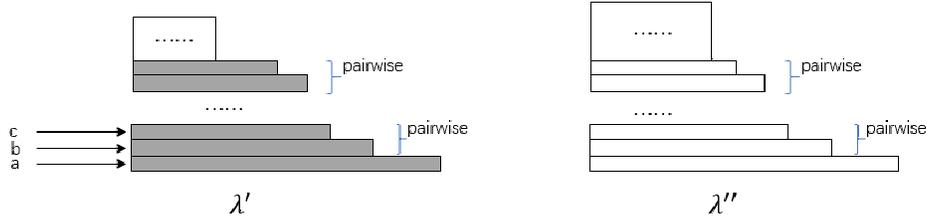}
  \end{center}
  \caption{ Partitions $\lambda^{'}$ and $\lambda^{''}$ are  in the $B_n$ and $C_n$ theories, respectively.}
  \label{f0}
\end{figure}
$(\lambda^{'}, \lambda^{''})$ is a  rigid semisimple   operator in the $B_n$ theory as shown in Fig.(\ref{f0}). The partition $\lambda$ is constructed   inserting rows of $\lambda^{'}$  into $\lambda^{''}$  one by one. This procedure would be helpful to understand the construction of the partition $\mu$.

\begin{figure}[!ht]
  \begin{center}
    \includegraphics[width=4.9in]{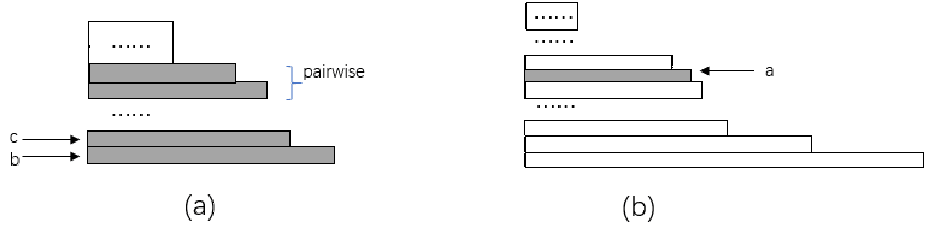}
  \end{center}
  \caption{ The first row $a$ of $\lambda^{'}$ is inserted into the partition $\lambda^{''}$.}
  \label{f1}
\end{figure}
\begin{figure}[!ht]
  \begin{center}
    \includegraphics[width=4.9in]{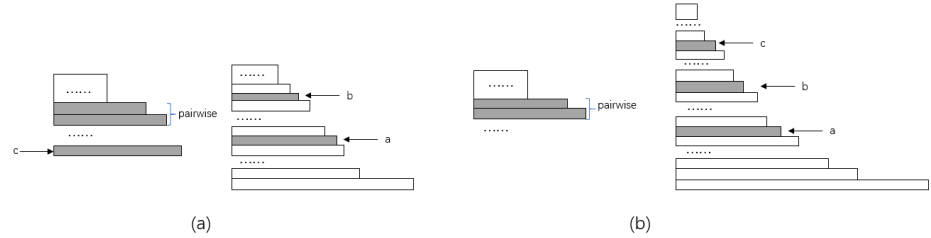}
  \end{center}
  \caption{ Pairwise  rows $b$ and $c$ of $\lambda^{'}$ are inserted into the partition $\lambda^{''}$.}
  \label{f2}
\end{figure}
The insertion of  the first row   $a$  is shown in Fig.(\ref{f1}).  The insertions of  $b$ and $c$ of a  pairwise rows  are shown in Figs.(\ref{f2})(a) and (b).
The row  $b$  can be placed   between or under a  pairwise rows, and these three rows form  the lower boundary of the block in the box as shown in Fig.(\ref{bl}). Similarly, the row   $c$  can be placed   between or above  a pairwise rows, and these three rows form   the  upper  boundary of the block in the box. The rows between the boundaries are pairwise rows of $\lambda^{''}$. Since the first row of each pairwise rows is even, they can be regard as the pairwise rows of a $B_n$ partition.   Under these constructions,  each block consist of pairwise rows of $\lambda^{'}$ and $\lambda^{''}$, which means the number of total boxes is even. These processes  continue until all the rows of $\lambda^{'}$ are  inserted into  $\lambda^{''}$.  The partition $\lambda=\lambda^{'}+\lambda^{''}$  is decomposed into several  blocks, consisting of types of blocks such as $I,II,III$ as shown in  Fig.(\ref{bl}). \textbf{We would refine this model in the end.}
\begin{figure}[!ht]
  \begin{center}
    \includegraphics[width=4.9in]{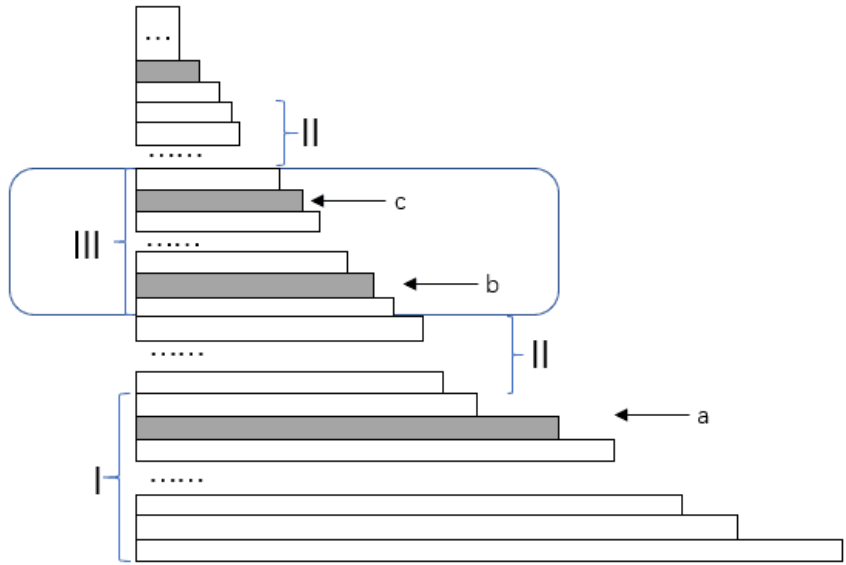}
  \end{center}
  \caption{ Partition $\lambda$ and  blocks.}
  \label{bl}
\end{figure}

\begin{flushleft}
{$III$ \textbf{type block with operators $\mu_{e11}$, $\mu_{e12}$, $\mu_{e21}$, and $\mu_{e22}$}}
\end{flushleft}

\begin{figure}[!ht]
  \begin{center}
    \includegraphics[width=4.9in]{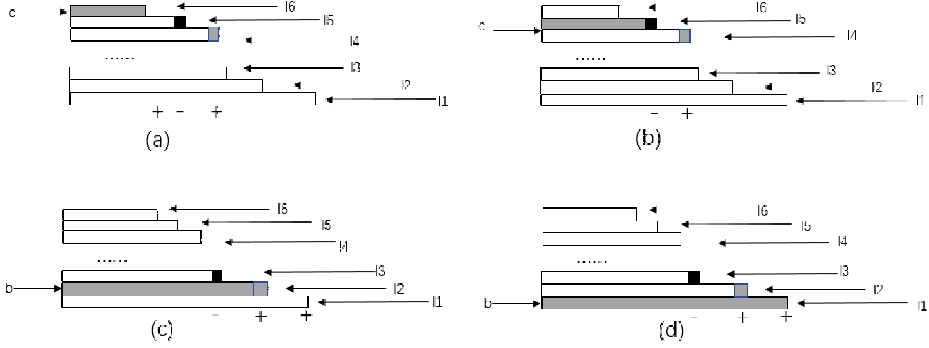}
  \end{center}
  \caption{ An even pairwise  rows $c$, $b$ of   $\lambda^{'}$ are inserted into the partition $\lambda^{'' }$. }
  \label{be}
\end{figure}
The fingerprint can be constructed by summing  the contribution of each block.  First we consider the block of type $III$ as shown in Fig.(\ref{bl}).
Let the  pairwise rows $b$ and $c$  of $\lambda^{'}$ in the block are even   as shown in Fig.(\ref{be}).
The row $c$ is the second row of the  pairwise rows, forming  the upper boundary of the block with an odd pairwise rows of $\lambda^{''}$.  The row $b$ is the first row of  the  pairwise rows,  forming the lower boundary of the block with an odd pairwise rows of $\lambda^{''}$. The number of boxes of rows   above the block is even. And the number of boxes  under the first row of the block is also even. \textit{So the sign of $p(i)$ corresponding to the last part of a low is determined by  the number of boxes  in the block before the part.}  Since the row $c$ inserted is even,   the number of boxes of  the upper boundary  is even. So the sign of the last part of each  row is the same as  that of  row of pairwise rows of the $B_n$ partition. And we  only need to determine the changes of the upper boundary and  lower boundary under the map $\mu$.

The first row and the last row will not change under the map $\mu$.
\begin{enumerate}
\item The height   of $l6$ is  even, so the sign of the last part of the row $l6$  is $'+'$.
\item The total number of  boxes  of each block is even, so the sign of the last  part of the block  is $'+'$ in Fig.(\ref{be}).
\end{enumerate}
These  results is right even though an odd pairwise are inserted into $\lambda^{''}$.

For   Fig.(\ref{bo})(a),  row $c$ is above a pairwise rows.
\begin{enumerate}
\item The height   of  $l5$ are  odd, so the sign of the last part of row $l5$ is $'-'$.
\item According to Fig.(\ref{p2}), the  sign of the last part of  row $l4$ is $'-'$.
\end{enumerate}
If the row $c$ is inserted  into the middle of the pairwise rows, we can reach to  Fig.(\ref{be})(b) similarly.

For   Fig.(\ref{bo})(c), the number of boxes above the row $l3$ is even. So the sign of the last parts of $l3$, $l2$,  and $l1$ are determined by the number of boxes  before the last part.
\begin{enumerate}
  \item The  height  of the row $l3$ is  odd, so the sign of the last part  $l3$ is $'-'$, which means  the last box  of the row is deleted under the map $\mu$.
  \item The height  of  the row  $l2$ is even,  so   the  sign of the last part of $l2$ is $'-'$, which means we should append a box as the last part of the row under the map $\mu$.
\end{enumerate}
 If the row $b$ is inserted into  the middle of the pairwise rows, we can  reach to  Fig.(\ref{be})(d) similarly.

According to  Fig.(\ref{be}), the first  and the last rows of the block do not change under the map $\mu$. The rows $l2$  and $l3$   change as an odd  pairwise rows in the $B_n$ theory as well as the rows $l4$  and $l5$.  If we replace the rows $l2$  and $l3$ by  even rows, they would change as an even  pairwise rows in the $B_n$ theory  as well as  the rows $l4$  and $l5$.

According to the places where the rows $b$ and $c$ are inserted into $\lambda^{'' }$, we define four operators $\mu_{e11}$, $\mu_{e12}$, $\mu_{e21}$, and $\mu_{e22}$ as shown in Fig.(\ref{ob}), corresponding to all  the combinations of the upper boundaries Fig.(\ref{bo})(a) Fig.(\ref{bo})(b) and the low boundaries  Fig.(\ref{bo})(c),  Fig.(\ref{bo})(d). For example, the upper boundary of the operators $\mu_{e11}$ is Fig.(\ref{be})(a) and the low boundary is Fig.(\ref{be})(d).

\begin{flushleft}
{$III$ \textbf{type block with operators $\mu_{o11}$, $\mu_{o12}$, $\mu_{o21}$, and $\mu_{o22}$ }}
\end{flushleft}

Next, we consider  an odd pairwise  rows in $\lambda^{'}$ is inserted   into  $\lambda^{'' }$, comprising   a block as shown in Fig.(\ref{bo}).
The row $c$ is the second row of the pairwise rows, forming  the upper boundary of the block with an even pairwise rows of $\lambda^{''}$.  The row $b$ is the first row of a pairwise rows,  forming the lower boundary of the block with an even pairwise rows of $\lambda^{''}$. So the numbers of boxes of  lower boundary  and upper boundary of block are odd.  The numbers of boxes   above the block are even. \textit{So the sign of the last box of a low  in the block is determined by  the number of boxes before the part.}

 \begin{figure}[!ht]
  \begin{center}
    \includegraphics[width=4.9in]{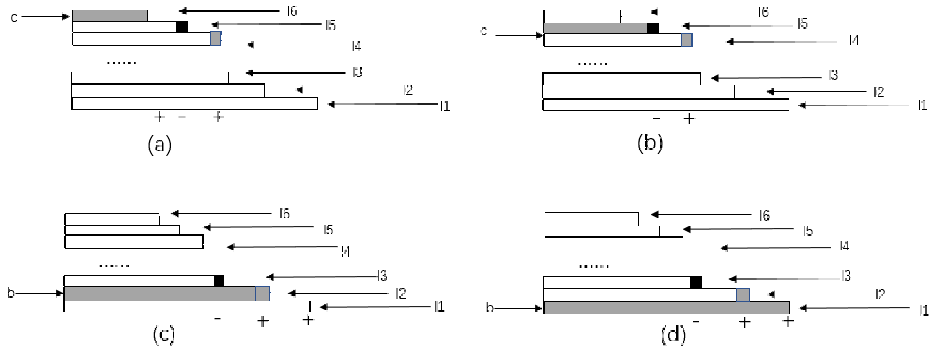}
  \end{center}
  \caption{ An odd pairwise  rows $c$, $b$ of $\lambda^{'}$ are inserted into the partition $\lambda^{'' }$. }
  \label{bo}
\end{figure}
 First, we determine the changes of the upper boundary  under the map $\mu$ for   Fig.(\ref{bo})(a).
 \begin{enumerate}
   \item The height   of the row $l6$ is  even, so the sign of the last box of row $l6$  is $'+'$.
   \item The number of boxes of $l6$ is odd.   The height   of the row $l5$ is  odd and   the number of boxes of the row $l5$  is even. So the sign of the last part of row $l5$  is $'-'$.
   \item According to Fig.(\ref{p2}), the  sign of the last box of  row $l4$ is $'-'$.
 \end{enumerate}
 If the row $c$ is inserted  into the middle of the pairwise rows, we can reach to  Fig.(\ref{bo})(b) similarly.

Next we determine changes of  pairwise rows of $\lambda^{'' }$ between the upper boundary and the lower boundary in the block.  Since  the number of boxes of   upper boundary of block is odd, the even  pairwise rows and the odd pairwise rows exchange roles compared with Fig.(\ref{bo}). So we should append a box as the last part of the first row of an even pairwise rows and delete the last box of the second one. While   rows of an odd pairwise rows do not change.

Finally  we   determine the changes of the   lower boundary under the map $\mu$ for   Fig.(\ref{bo})(c). The number of boxes above the row $l3$ is odd since the inserted row $c$ in the upper boundary is odd.
The sign of the last boxes of $l3$, $l2$,  and $l1$ are determined by the number of boxes  before the corresponding part.
\begin{enumerate}
  \item  The height  of the row $l3$ is  odd  and the length  is even. So the sign of the last box of $l3$ is $'-'$.
  \item The height of  the row  $l2$ is even  and the length  is odd,  so the   the  sign of the last box of $l2$ is $'-'$.
\end{enumerate}
If the row $b$ is inserted into  the middle of the pairwise rows, we can  reach to the  Fig.(\ref{be})(d) similarly.  Since the block consist of pairwise rows of $\lambda^{' }$  and $\lambda^{'' }$, the total number boxes of the block is even, which means the sign of the last  box of the block  is $'+'$ in  Fig.(\ref{bo})(a), (b), (c), and (d).

According to  Fig.(\ref{bo}), the first row and the last row of the block do not change under the map $\mu$. The rows $l2$  and $l3$   change as an odd  pairwise rows in the $B_n$ theory as well as the rows $l4$  and $l5$.  If we replace the rows $l2$  and $l3$ by two odd rows, they change as an even  pairwise rows in the $B_n$ theory  as well as  the rows $l4$  and $l5$. Summary, compared with Fig.(\ref{be}),  nothing changes in Fig.(\ref{bo}) excepting the even  pairwise  and the odd pairwise exchange roles. Similarly, we can define   operators $\mu_{o11}$, $\mu_{o12}$, $\mu_{o21}$, and $\mu_{o22}$.

\begin{flushleft}
{\textbf{Fingerprint of $III$ type block}}
\end{flushleft}

\begin{figure}[!ht]
  \begin{center}
    \includegraphics[width=4.9in]{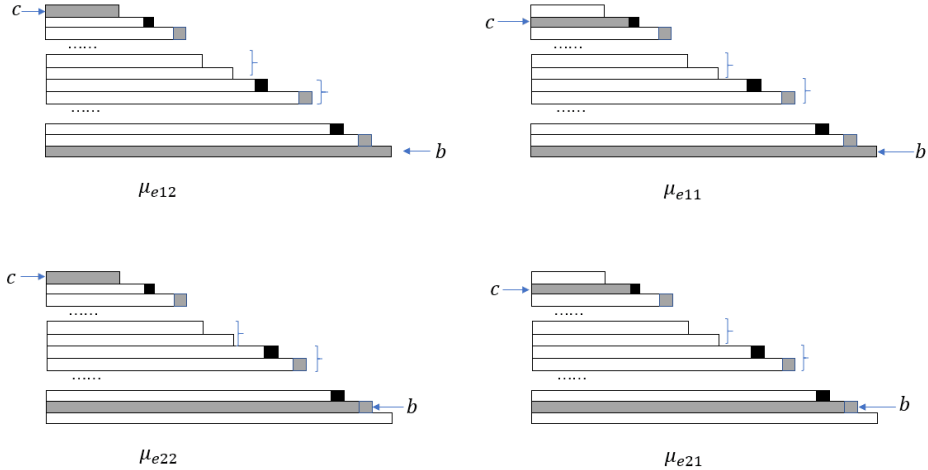}
  \end{center}
  \caption{ $III$ type block.}
  \label{ob}
\end{figure}
  Now, we calculate the fingerprints of the blocks as shown in   Fig.(\ref{ob}).
A comparison  of  a $B_n$ partition  and the $III$ type block  is made previously. The $III$ type block have the exactly the same structure with the $B_n$ partition under the map $\mu$, excepting the first and the last row. For example, the condition '$(iii)_{SO} \quad \lambda^{'}_{i} \quad\textrm{is odd}$'   can not be  satisfied for  even parts of the block since the height of the inserted row $b$ in $\lambda^{'}$ is even. For the pairwise rows of $\lambda^{''}$ in the block, the  height of the first row   is even.  So  we can calculate  the fingerprint  of  the $III$ type blocks  using the formulas in Section \ref{fu}. Since the first row change nothing under the map $\mu$, we should use the  formula (\ref{fbfbfbfb}) to calculate the fingerprint of the last part of the block. The row following the last row of a block do not change under the map $\mu$. But the height satisfy the condition '$(iii)_{SO} \quad \lambda^{'}_{i} \quad\textrm{is odd}$'. So the fingerprint can be calculated by the formula \ref{fi2}.
\begin{figure}[!ht]
  \begin{center}
    \includegraphics[width=4.9in]{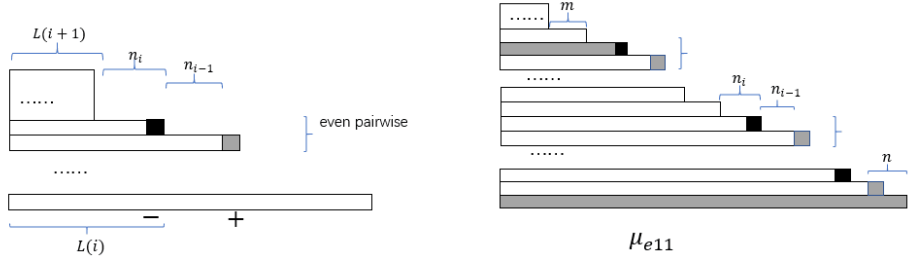}
  \end{center}
  \caption{ A comparison    between  a $B_n$  partition under the map $\mu$ and the  block $\mu_{e11}$.}
  \label{ccb}
\end{figure}

\begin{flushleft}
\textbf{$II$ type block}
\end{flushleft}
Next we  calculate the fingerprint of  $II$ type blocks  as shown in   Fig.(\ref{cf}). The height of the  first row of the  block is odd and all rows are in  pairwise pattern, which are exactly the same with a $C_n$ partition.      So the   fingerprint of this block can be calculated by using the formula (\ref{cfi}). We define the  operator $\mu_{II}$ corresponding to image of the block under the map $\mu$.
\begin{figure}[!ht]
  \begin{center}
    \includegraphics[width=4.9in]{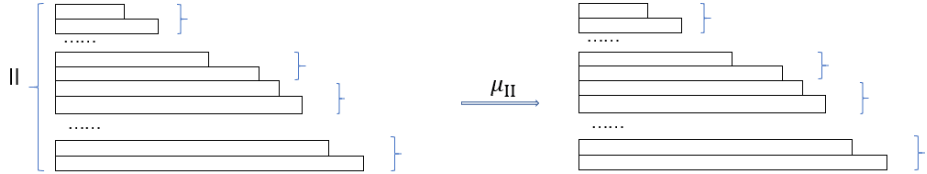}
  \end{center}
  \caption{   $II$ type block of $\lambda$ and its image under the map $\mu$.}
  \label{cf}
\end{figure}

\begin{flushleft}
\textbf{$I$ type block}
\end{flushleft}

For $I$ type block,  the first row $a$ of $\lambda^{'}$ is inserted into  $\lambda^{''}$ as shown in   Fig.(\ref{bi}) or  the first row  of $\lambda^{''}$ is inserted into  $\lambda^{'}$ .
 The row $a$is above an even pairwise rows in Fig.(\ref{bi})(a) and  is between  an even pairwise rows in Fig.(\ref{bi})(b).
 All  the rows are in pairwise pattern except the first  row  of $\lambda^{'}$ which  is odd and the first row of $\lambda^{''}$ which is even according to Proposition \ref{Pb} and \ref{Pc}. So the number of the boxes of the $I$ type block is odd.

Under the map  $\mu$, the images of the first row and the last row of $I$ type block  are independent of the position of $a$    which are  similar to $III$ type block.
\begin{itemize}
  \item Since the number of the boxes of the $I$ type block is odd, the sign of the last part of the first row is $'-'$, which means we should delete a box at the end of the first row according to Lemma \ref{ff}.
  \item Since  the number of the boxes above $I$ type block is even and the height of the last row is even, the sign of the last box of the last row  is $'+'$, which means the last row does not change under the map $\mu$. \end{itemize}

For the other rows, the arguments for the partition $\mu$ are exactly the same with the blocks in  Fig.(\ref{bo}). Comparing   Fig.(\ref{bi})(a) with  Fig.(\ref{bo})(a), we define the operator  $\mu_{o2}$ corresponding to Fig.(\ref{bi})(a). Comparing   Fig.(\ref{bi})(b) with  Fig.(\ref{bo})(b), we can define the operator  $\mu_{o1}$ corresponding to Fig.(\ref{bi})(b).
\begin{figure}[!ht]
  \begin{center}
    \includegraphics[width=4.9in]{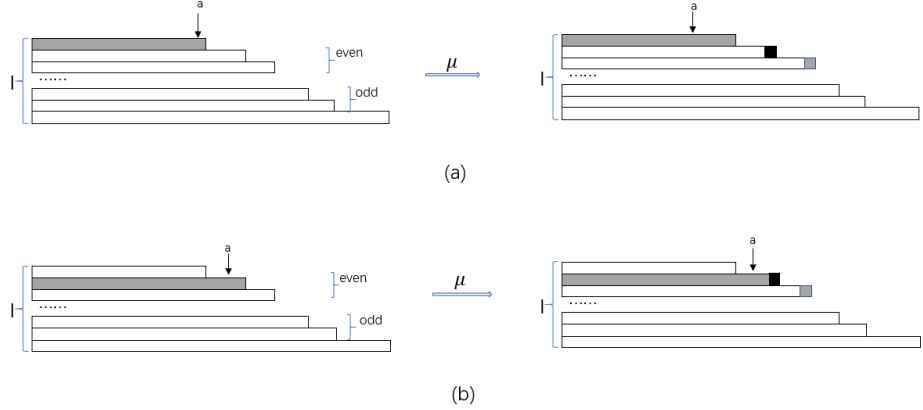}
  \end{center}
  \caption{ The first   $I$ type block of $\lambda$ and its image under  $\mu$.}
  \label{bi}
\end{figure}

If the first row of $\lambda^{''}$ is inserted into  $\lambda^{'}$ in the $I$ type block, the rows of  $\lambda^{'}$ are represented by white rows and the rows of  $\lambda^{'}$ are represented by gray rows. Then  the $I$ type block is exactly  the same as  shown in Fig.(\ref{bi}). The row  $a$ is even, which is  the first row of  $\lambda^{''}$. Comparing   Fig.(\ref{bi})(a) with  Fig.(\ref{ob})(a), we  define the operator  $\mu_{e1}$ corresponding to Fig.(\ref{bi})(a). Comparing   Fig.(\ref{bi})(b) with  Fig.(\ref{ob})(b), we define the operator  $\mu_{e2}$ corresponding to Fig.(\ref{bi})(b).

With the partition $\mu$ of the $I$ type block, the fingerprint can be calculated   using the formulas of  the fingerprint  of the $III$ type block directly.

\begin{flushleft}
\textbf{Special blocks}
\end{flushleft}

\begin{figure}[!ht]
  \begin{center}
    \includegraphics[width=4.9in]{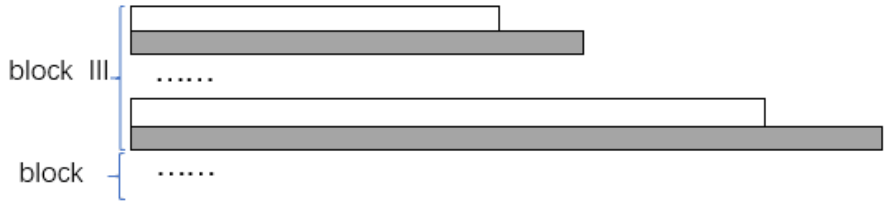}
  \end{center}
  \caption{ Blocks of  $\lambda$.}
  \label{ebo}
\end{figure}
 We ignore one situation in the above discussions.    If the first row of a block is in the partition $\lambda^{'}$ and  the second one  is  in the partition $\lambda^{''}$, then the row under the block  would be in another block as shown in  Fig.(\ref{ebo}). However, if the first two rows of a block are in the partition $\lambda^{'}$, the situation become complicated as shown in Fig.(\ref{eb}). The fingerprint of the $III$ type block  can be still calculated by previous formulas except the first row. We  determine the block $S$ by including one row in $\lambda^{'}$ and even number of rows in $\lambda^{''}$ as little as possible.
\begin{figure}[!ht]
  \begin{center}
    \includegraphics[width=4.9in]{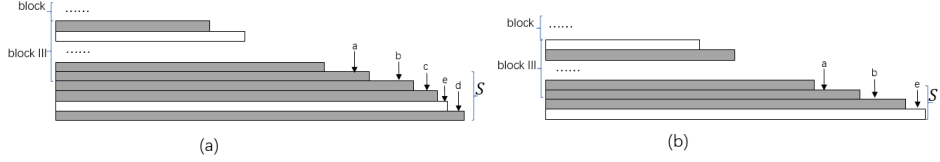}
  \end{center}
  \caption{ $S$ type block.}
  \label{eb}
\end{figure}
Unfortunately, we could not use the operators $\mu_{e11}$, $\mu_{e12}$, $\mu_{e21}$,  $\mu_{e22}$ and their odd versions to  determine  the images of  the blocks $S$ under the map $\mu$. So we must calculate  the partition $\mu$ of block $S$ case by case. Since the block $S$ and the following $III$ type block   are consist of pairwise rows of  $\lambda^{'}$ and $\lambda^{''}$,  the sign of the last box of the  first row of the block $S$ is $'+'$, which means this first row would not  change under the map $\mu$.

\subsection{$D_n$ theory}
For a  rigid semisimple   operator $(\lambda^{'}, \lambda^{''})$ in the  $D_n$ theory, the  fingerprint can be constructed  by  the same  arguments as that of    a partition in the $B_n$ theory.  The $I$ type block of   $\lambda=\lambda^{'}+\lambda^{''}$ need to be paid more attentions since both $\lambda^{'}$ and $\lambda^{''}$  are in the $D_n$ theory and their  first rows    are even.  We  define operators $\mu_{e1}$ and  $\mu_{e2}$  corresponding to Fig.(\ref{bi})(a) and   Fig.(\ref{bi})(b) for the $I$ type block, respectively.

Since the total number of the boxes of $\lambda$ is even, the sign of the last part of $\lambda$ is $'+'$, which means nothing change under the map $\mu$ for the first row of $\lambda$.
The fingerprints of other parts can be calculated by the same exactly formulas as in the $B_n$ theory.

\subsection{$C_n$ theory}
 \begin{figure}[!ht]
  \begin{center}
    \includegraphics[width=4.9in]{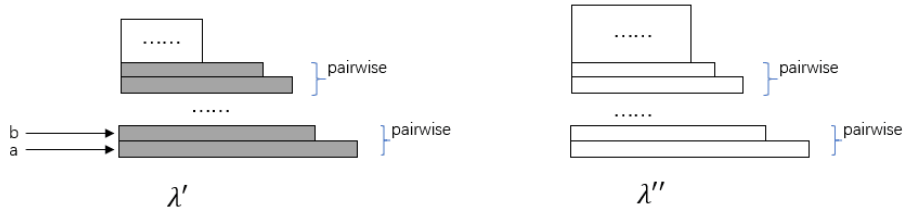}
  \end{center}
  \caption{Partitions $\lambda^{'}$ and $\lambda^{''}$ in the $C_n$ theory }
  \label{c0}
\end{figure}
For a  rigid semisimple   operator $(\lambda^{'}, \lambda^{''})$ in the  $C_n$ theory,  both $\lambda^{'}$ and $\lambda^{''}$  are in the $C_n$ theory.  The first two rows of   $C_n$ partitions are  in  pairwise pattern as shown in Fig.(\ref{c0}).
Compared   Fig.(\ref{ccc0}) with Fig.(\ref{bl}), the  construction of fingerprint  is even simpler than  that of  the $B_n$ and $D_n$ theories  because no $I$ type block appears in $\lambda$. And the  fingerprints of other parts can be calculated by the same exactly formulas as in the $B_n$ and $D_n$ theories.
\begin{figure}[!ht]
  \begin{center}
    \includegraphics[width=4.9in]{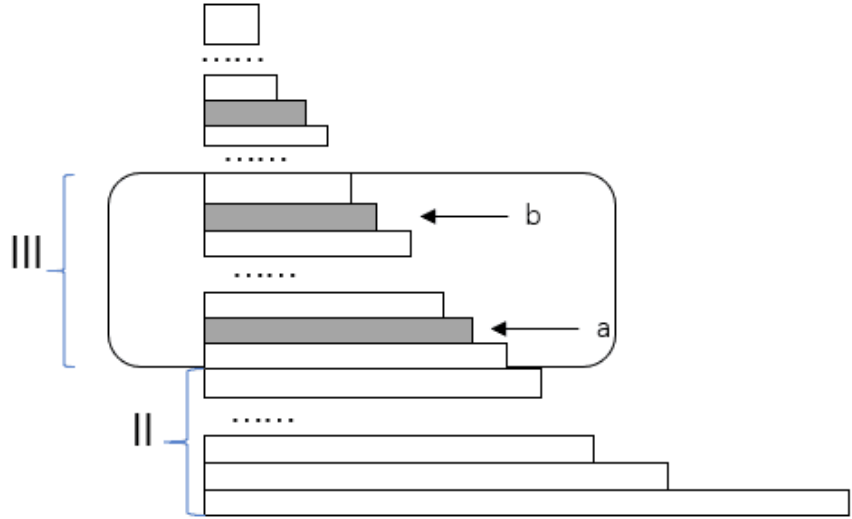}
  \end{center}
  \caption{Partition $\lambda$ }
  \label{ccc0}
\end{figure}

\section{Acknowledgments}
 The work of Qiao Wu has been supported by Key Research Center of Philosophy and Social Science of Zhejiang Province: Modern Port Service Industry and Creative Culture Research Center(15JDLG02YB). This work of Bao Shou was supported by a grant from  the Postdoctoral Foundation of Zhejiang Province.




\end{document}